\documentclass{article}
\usepackage[utf8]{inputenc}
\usepackage{amsmath}
\usepackage{amsfonts}
\usepackage{amssymb}
\usepackage{amsthm} 
\usepackage{csquotes}
\usepackage[english]{babel}
\usepackage{lmodern}
\usepackage[shortlabels]{enumitem}

\newtheorem{theorem}{Theorem}
\newtheorem*{theorem*}{Theorem}

\newtheorem{corollary}{Corollary}

\newcommand{\di}{\mbox{d}}
\newcommand{\N}{\mbox{$\mathcal{N}$}}
\newcommand{\h}{\mbox{$\mathcal{H} $ }}

\title{On the classification of contact metric $(k,\mu)$--spaces via tangent hyperquadric bundles}
\author{E. Loiudice and A. Lotta}
\date{}

\begin{document}
 \maketitle

 \begin{abstract}
 We classify locally the contact metric $(k,\mu)$--spaces whose Boeckx invariant is $\le -1$ as tangent hyperquadric bundles
  of Lorentzian space forms.
 \end{abstract}
 
{\small

}

\section{Introduction}

Contact $(k,\mu)$--spaces constitute a relevant class of objects studied in contact metric geometry; these spaces were
introduced by Blair, Koufogiorgos and Papantoniou 
in \cite{Blair K P} as a generalization of Sasakian manifolds.
 Indeed, a \emph{contact metric $(k,\mu)$ manifold} is a contact metric manifold $(M,\varphi, \xi,\eta,g)$ such that
 \begin{equation}
 \label{kmu}
  R(X,Y)\xi=k(\eta(Y)X-\eta(X)Y)+\mu(\eta(Y)hX-\eta(X)hY),
 \end{equation}
where $X,Y\in \mathfrak{X}(M)$, $k$, $\mu$ are real numbers and $h:=\frac{1}{2}\mathcal{L}_{\xi}\varphi$. Here $\mathcal{L}_{\xi}\varphi$ denotes the Lie derivative of $\varphi$ in the direction of $\xi$. Recall that Sasakian manifolds are characterized by the above equation with $k=1$ and $h=0$.

\smallskip
Looking at contact metric manifolds as strongly pseudo-convex (almost) $CR$ manifolds, Dileo and Lotta \cite{DileoLotta} showed that
the $(k,\mu)$--condition is equivalent to the local $CR$-symmetry with respect to the Webster
metric $g$, according to the general notion introduced by Kaup and Zaitsev in \cite{Kaup Zaitsev} (see also section \ref{Sec CR}).
In this context, another characterization was given in terms of the parallelism of  the Tanaka--Webster curvature and torsion \cite{Boeckx-Cho}. 

\smallskip
Boeckx gave a crucial contribution to the problem of classifying these manifolds; 
after showing that every non-Sasakian contact $(k,\mu)$-space is locally homogeneous and strongly locally $\varphi$--symmetric \cite{Boeckx locally phi-symmetric},  
in \cite{Boeckx} he defined a scalar invariant $I_M$ which completely determines a contact $(k,\mu)$--space $M$ locally up 
to equivalence and up to a $D$-homotetic deformation of its contact metric structure. 

\smallskip
A standard example is the tangent sphere bundle $T_1M$ of a Riemannian manifold $M$ with constant sectional curvature $c\not =1$.
Being an hypersurface of $TM$, which is equipped with a natural
strictly almost K\"ahler structure $(J,G)$, where $G$ is the Sasaki metric, $T_1M$ inherits a standard
contact metric structure (for details, see for instance
\cite{Blair2010Book}). In particular, the Webster metric $g$ of $T_1M$ 
is a scalar multiple of $G$. The corresponding Boeckx invariant is given by:
$$
I_{T_1M}=\frac{1+c}{|1-c|}.
$$
Hence, as $c$ varies in $\mathbb{R} \smallsetminus \{1\}$, $I_{T_1M}$ assumes all the real values strictly greater than $-1$.

The case $I\leqslant -1$ seems to lead to models of different nature.
Namely, Boeckx found examples of contact metric $(k,\mu)$ spaces,  for every value of the invariant $I\leqslant -1$, namely
a two parameter family of  Lie groups with a left--invariant contact metric structure. However,
he gave no geometric description of these examples.

\smallskip
The purpose of this paper is to show that, actually, one can construct the models with $I\leqslant -1$
simply by replacing a Riemannian space form $(M,g)$ with a Lorentzian  one,
taking instead of $T_1M$ the so-called tangent hyperquadric bundle:
$$T_{-1}M=\{(p,v)\in TM:\quad g(v,v)=-1\}.$$
Indeed, the formula for the Boeckx invariant changes as follows:
$$I_{T_{-1}M}=\frac{c-1}{|c+1|},$$
where $c$ varies in $\mathbb{R}\smallsetminus\{-1\}$, so that for $c\leqslant 0$, these examples cover all possible values of the Boeckx
invariant in $(-\infty,-1]$.

\smallskip
We remark that, as in the Riemannian case, $T_{-1}M$ is again a strictly pseudo-convex $CR$ hypersurface of $(TM,J)$ (see also \cite{Perrone}
for a recent study of these manifolds from the point of view of $CR$ geometry).
However, in this case the Webster metric $g$
is no longer a scalar multiple of the (semi-Riemannian) Sasaki metric of $TM$.

 \section{Preliminaries}
\subsection{Contact metric $(k,\mu)$ manifolds}
 
In this section we recall some basic results concerning the class of contact metric manifolds under consideration. As a general reference
on contact metric geometry, we refer the reader to Blair's book \cite{Blair2010Book}.
 
\begin{theorem*}[\cite{Blair K P}]\label{th (k,mu) Blair}
 Let $(M,\varphi, \xi,\eta,g)$ be a contact metric $(k,\mu)$ manifold. Then necessarily $k\leqslant 1$. Moreover, if $k=1$ then $h=0$ and $(M,\varphi, \xi,\eta,g)$ is Sasakian. If $k<1$, the contact metric structure is not Sasakian and $M$ admits three mutually orthogonal integrable distributions $\mathcal{D}(0)$, $\mathcal{D}(\lambda)$ and $\mathcal{D}(-\lambda)$ corresponding to the eigenspaces of $h$, where $\lambda=\sqrt{1-k}$.
\end{theorem*}

Non-Sasakian contact metric $(k,\mu)$ manifolds was completely classified by Boeckx in \cite{Boeckx}. We have that $k<1$ and the real number 
$$
I_M=\frac{1-\frac{\mu}{2}}{\sqrt{1-k}},
$$
is an invariant for the $(k,\mu)$ structure; moreover: 
\begin{theorem*}[\cite{Boeckx}]
Let $(M_i, \varphi_i, \xi_i, \eta_i, g_i )$, $i=1, 2$, be two non-Sasakian $(k_i, \mu_i)$-spaces of the same dimension. Then $I_{M_1}=I_{M_2}$ if and only if, up to a $D$-homothetic
transformation, the two spaces are locally isometric as contact metric spaces. In particular, if both spaces are simply connected and complete, they are globally isometric up to a $D$-homothetic transformation.
\end{theorem*}

\medskip

Finally we recall how the Boeckx invariant $I_M$ of a non-Sasakian $(k,\mu)$ manifold is linked with the behavior of the Pang invariants of the Legendre foliations determined by $\mathcal{D}(\lambda)$ and $\mathcal{D}(-\lambda)$.
The Pang invariant of a Legendre foliation $\mathcal{F}$ on a contact manifold $(M,\eta)$ is the symmetric tensor:
\begin{equation}\label{Pang inv def}
 \varPi_{\mathcal{F}}(X,Y):=-(\mathcal{L}_X\mathcal{L}_Y\eta)(\xi)=2\di \eta([\xi,X],Y),
\end{equation}
where $X,Y$ are vectors fields tangent to $\mathcal{F}$ (cf. \cite{Pang}).
The Legendre foliation $\mathcal{F}$ is called positive, negative or flat according to the circumstance that the bilinear form $\varPi_{\mathcal{F}}$ is positive definite, negative definite or vanishes identically, respectively.

In our case the explicit expressions of $\varPi_{\mathcal{D}(\lambda)}$ and $\varPi_{\mathcal{D}(-\lambda)}$ are (see \cite[p. 127]{Blair2010Book}
or \cite{MinoDiTerlizzi}):
\begin{equation}\label{Pang inv 1}
  \varPi_{\mathcal{D}(\lambda)}=\frac{(\lambda+1)^2-k-\mu \lambda}{\lambda}g_{\eta}|_{\mathcal{D}(\lambda) \times \mathcal{D}(\lambda)},
  \end{equation}
  \begin{equation}\label{Pang inv 2}
 \varPi_{\mathcal{D}(-\lambda)}=\frac{-(\lambda-1)^2+k-\mu \lambda}{\lambda}g_{\eta}|_{\mathcal{D}(-\lambda) \times \mathcal{D}(-\lambda)}.
\end{equation}
 Using the previous equations one gets (see \cite{Mino}):
 \begin{theorem*}
  Let $(M,\varphi, \xi,\eta,g)$ be a non-Sasakian contact metric $(k,\mu)$ manifold. Then one of the following must hold:
  \begin{enumerate}[(a)]
   \item both $\mathcal{D}(\lambda)$ and $\mathcal{D}(-\lambda)$ are positive definite;
   \item $\mathcal{D}(\lambda)$ is positive definite and $\mathcal{D}(-\lambda)$ is negative defined;
   \item both $\mathcal{D}(\lambda)$ and $\mathcal{D}(-\lambda)$ are negative definite;
   \item $\mathcal{D}(\lambda)$ is positive definite and $\mathcal{D}(-\lambda)$ is flat;
   \item $\mathcal{D}(\lambda)$ is flat and $\mathcal{D}(-\lambda)$ is negative defined.
  \end{enumerate}
  Furthermore, $M$ belongs to the class (a), (b), (c), (d), (e) if and only if $I_M>1$, $-1<I_M<-1$, $I_M<-1$, $I_M=1$, $I_M=-1$, respectively.
 \end{theorem*}

\subsection{Locally symmetric pseudo-Hermitian manifolds}\label{Sec CR}
Let $M^{n+k}$ be a smooth manifold. A \emph{partial complex structure} of $CR$-codimension $k$ is a pair $(HM,J)$ where $HM$ is a smooth real subbundle of real dimension $2n$ of the tangent bundle $TM$, and $J$ is a smooth bundle isomorphism $J:HM\rightarrow HM$ such that $J^2=-I$. 

An \emph{almost $CR$ structure} on $M$ is a partial complex structure $(HM,J)$ on $M$ satisfying:
\begin{equation*}
 [X,Y]-[JX,JY]\in \h,
\end{equation*}
for every $X,Y\in \h$, where \h  denotes the module of all the smooth sections of $HM$.
If moreover the following equation 
\begin{equation}\label{integrability condition}
 [JX,JY]-[X,Y]-J([JX,Y]+[X,JY])=0,
\end{equation}
holds for every $X,Y\in \h$, then $(HM,J)$ is said to be a \emph{$CR$ structure} on $M$, and $(M, HM,J)$ is a \emph{$CR$ manifold}.

If $(HM,J)$ is an almost $CR$ structure of $CR$-codimension $1$ on a orientable manifold $M$, 
one can also represent $HM$ as
$$
HM=\ker (\eta),
$$
where $\eta$ is a globally defined nowhere vanishing one form.
If, for some choice of $\eta$, the corresponding \emph{Levi form} 
\begin{equation*}
 L_{\eta}(X,Y):=-\di \eta (X,JY), \quad X,Y\in \h
\end{equation*}
is positive definite, then the almost $CR$ structure  $(HM,J)$ is a called 
\emph{strongly pseudo-convex},  and we shall refer to $(M,HM,J,\eta)$ as a \emph{strongly pseudo-convex almost $CR$ manifold}. When 
the underlying almost $CR$ structure is also integrable, $M$ is usually termed a 
{\em pseudo-Hermitian manifold}. See for example \cite{LibroDragomirTomassini} for more information on this subject.

\smallskip
It is well known that any strongly pseudo-convex structure $(HM,J,\eta)$ on $M$ determines a contact metric structure $(\varphi, \xi,\eta,g_{\eta})$.
Indeed, $\eta$ is a contact form, so that there exists a unique nowhere vanishing globally defined vector field $\xi$ transverse to $HM$ (the Reeb
vector field), such that
\begin{equation*}
 \eta(\xi)=1, \quad \di \eta (\xi, X)=0
\end{equation*}
for every $X\in \mathfrak{X}(M)$. The Levi form $L_{\eta}$ and the bundle isomorphism $J$ can be canonically extended respectively to a Riemannian metric $g_{\eta}$, called the \emph{Webster metric}, and to a $(1,1)$-tensor field $\varphi$:
\begin{equation*}
 g_{\eta}(X,Y):=L_{\eta}(X,Y), \quad g_{\eta}(X,\xi)=0,\quad g_{\eta}(\xi, \xi)=1,
\end{equation*}
\begin{equation*}
 \varphi X:=JX,\quad \varphi \xi :=0,
\end{equation*}
where $X,Y\in \h$. One can check that $(\varphi, \xi,\eta,g_{\eta})$ is a contact metric structure on $M$ in the sense of \cite{Blair2010Book}. 

Conversely, if $(\varphi, \xi,\eta,g)$ is a contact metric structure on $M$, then setting
\begin{equation*}
 HM:=\text{Im}(\varphi), \quad J:=\varphi|_{HM},
\end{equation*}
one gets  a  strongly pseudo-convex almost $CR$ structure,  whose Webster metric $g_{\eta}$ coincides with $g$.

\medskip

Let $(M,HM,J,\eta)$ be a strongly pseudo-convex almost $CR$ manifold.
A \emph{$CR$ symmetry} at a point $p\in M$ is a $CR$ diffeomorphism
\begin{equation*}
 \sigma:M\rightarrow M,
\end{equation*}
which is also an isometry with respect to the Webster metric $g_{\eta}$, and such that 
\begin{equation*}
 (\di \sigma )_p|_{H_pM}=-Id_{H_pM}.
\end{equation*}
If $M$ admits a $CR$ symmetry at $p$ for every $p\in M$, then $M$ will be called a {\em symmetric pseudo-Hermitian manifold},
according to the general terminology introduced by Kaup and Zaitsev in  \cite{Kaup Zaitsev}.
Since the symmetry at $p$ in uniquely determined (cf. Theorem 3.3 in \cite{Kaup Zaitsev}), it makes sense also
to define {\em locally symmetric} pseudo-Hermitian manifolds in a natural manner. Observe that, since the  local $CR$ symmetries are $CR$ maps, 
for these manifolds the integrability condition \eqref{integrability condition} is automatically satisfied.

Finally, we recall that it was showed in \cite[Theorem 3.2]{DileoLotta} that a non Sasakian contact metric manifold satisfies the $(k,\mu)$ condition \eqref{kmu}
if and only if it is a locally symmetric pseudo-Hermtian manifold.

\subsection{Tangent bundles and tangent hyperquadric bundles}\label{sec tangent bundle}
Here we recall some notions and properties on the tangent bundle of a manifold. The definition and some properties of the tangent hyperquadric bundle of a Lorentzian manifold will be also recalled.

 Let $M$ be a smooth manifold. The \emph{vertical lift} $X^V$ of a vector field $X$ on $M$, is the vector field on $TM$ defined by
 \begin{equation*}
  X^V \omega =\omega(X)\circ \pi,
 \end{equation*}
 where $\omega$ is any $1$-form on $M$ and $\pi:TM \rightarrow M$ is the canonical projection. Furthermore, if $D$ is an affine connection on $M$, the \emph{horizontal lift} of $X$ with respect to $D$, is defined by
 \begin{equation*}
  X^H \omega = D_X\omega,
 \end{equation*}
where $\omega$ is any $1$-form on $M$. The local expression of $X^H$ with respect to the local coordinates system $(q^i, v^i)$ on $TM$ associated to a local system of coordinates $(x^i)$ on $M$ is: 
\begin{equation}\label{local expression X^H}
 X^H=X^i \frac{\partial}{\partial q^i}-X^i v^j \Gamma^k_{i j}\frac{\partial}{\partial v^k}.
\end{equation}
We denote by $H_t$ and $V_t$ the span of the horizontal and vertical lifts at $t\in TM$  respectively. We have that:
\begin{equation*}
 T_t(TM)= H_t \oplus V_t.
\end{equation*}
The \emph{canonical vertical vector field} $\mathcal{N}$ on $TM$ and the \emph{geodesic flow} $\zeta$ on $TM$ are defined by:
\begin{equation*}
 \mathcal{N}_{t}=u^V_{t}, \quad \zeta_{t}=u^H_{t},\quad  t=(p,u)\in TM.
\end{equation*}

The tangent bundle of an affine manifold $(M,D)$ admits a canonical almost complex structure $\tilde{J}:TTM\rightarrow TTM$ such that:
\begin{equation*}
 \tilde{J} X^H=X^V, \quad \tilde{J}X^V=-X^H,\quad X\in \mathfrak{X} (M).
\end{equation*}

Observe that: 
$$\tilde{J} \mathcal{N}=-\zeta, \quad \tilde{J} \zeta=\mathcal{N}.$$

For the Lie brackets between horizontal and vertical lifts of the vector fields $X, Y$ on $M$, the following formulas hold (see \cite{Blair2010Book}):
\begin{equation*}
  [X^V,Y^V]=0, \quad [X^H, Y^V]=(D_XY)^V,
\end{equation*}
\begin{equation}\label{brackets}
   [X^H,Y^H]_t=[X,Y]^H_t-(R(X,Y)u)_t^H, 
\end{equation}
where $R$ denotes the curvature tensor of $D$ on $M$.

\medskip
In all that follows, we consider a Lorentzian manifold $(M,g)$. The \emph{Sasaki metric} $\tilde{G}$ on $TM$ is defined according to:
\begin{equation*}
\tilde{G}_t(X^H,Y^H)=g_p(X_p,Y_p),\quad \tilde{G}_t(X^V,Y^V)=g_p(X_p,Y_p), \quad \tilde{G}_t(X^H,Y^V)=0,
\end{equation*}
where $X,Y \in\mathfrak{X}(M)$, $t=(p,u)\in TM$, and $X^H$, $Y^H$ are the horizontal lifts of $X,Y$ with respect to the Levi-Civita connection of $g$. Observe that the Sasaki metric $\tilde{G}$ has index $2$ (see \cite{Perrone Dragomir} for more details). 
It is known that the $1$-form on $TM$:
\begin{equation*}
\begin{aligned}
 \beta_t(\tilde{X}_t):=\tilde{G}_t(\tilde{X}_t,u_t^H) =g_p(\pi_{\star}\tilde{X}, u), \quad \tilde{X}\in \mathfrak{X}(TM), \; t=(p,u)\in TM,\\        
\end{aligned}                     
\end{equation*}
 satisfies
\begin{equation}\label{beta}
 2 \di \beta (\tilde{X},\tilde{Y})=\tilde{G}(\tilde{X},\tilde{J}\tilde{Y}),
\end{equation}
for every $\tilde{X},\tilde{Y} \in \mathfrak{X}(TM)$  (see for instance \cite[p. 171]{Blair2010Book} or \cite{Perrone Dragomir}), so that 
$(TM,\tilde{J},\tilde{G})$ is an indefinite almost K\"ahler manifold.

\medskip 

Now we consider the tangent hyperquadric bundle
 \begin{equation*}
  T_{-1}M := \{(p,u)\in TM \; |\; g_p(u,u)=-1\},
 \end{equation*}
 which is an orientable hypersurface of $TM$, being $\mathcal{N}$ a unit normal vector field to $T_{-1}M$.
 We have that:
 $$
 \tilde{G}_t(\N_t, \N_t)=-1, \quad H_t\subset T_t(T_{-1}M),$$
 \begin{equation*}
 T_t(T_{-1}M)=\{X^H_t+Y^V_t \;|\; X,Y\in T_pM, \; g_p(Y,u)=0 \},
\end{equation*}
 for every $t=(p,u)\in T_{-1}M$.
 Being a hypersurface of $(TM,\tilde{J})$, $T_{-1}M$ inherits a canonical partial complex structure $(H(T_{-1}M),J)$, where
\begin{equation*}
 H(T_{-1}M):=\{X\in T(T_{-1}M) \;|\; \tilde{J}X\in T(T_{-1}M) \},
\end{equation*}
and 
$$
J:H(T_{-1}M)\rightarrow H(T_{-1}M),
$$
is the restriction of the almost complex structure $\tilde{J}$. 
Observe that for every $t=(p,u)\in T_{-1}M$:
\begin{equation*}
\begin{aligned}
 H_t(T_{-1}M)=&\{X^H_t+Y^V_t \;|\; X,Y\in T_pM, \; g_p(X,u)=0, \; g_p(Y,u)=0\}\\
             =&\{X^O_t+Y^T_t \;|\; X,Y\in T_pM \}
 \end{aligned}
\end{equation*}
where, for every $X\in\mathfrak{X}(M)$, we introduce the following vector fields tangent
to $T_{-1}M$:
$$
X^O:=X^H+\tilde{G}(X^V,\N )\zeta,
$$
$$
X^T:=X^V + \tilde{G}(X^V,\N )\N.
$$

\medskip

Finally, we consider the $1$-form $\eta:= \frac{1}{2} \beta$ on $T_{-1}M$, whose kernel
is $H(T_{-1}M)$. 
Equation \eqref{beta} implies that the Levi form $L_\eta$ 
 is positive definite and the Reeb vector field $\xi$ of $\eta$ is 
 \begin{equation}
 \label{Reeb}
 \xi_t=-2 \zeta_t, \quad t\in T_{-1}M.
 \end{equation}
Hence $(H(T_{-1}M),J, \eta)$ is a strongly pseudo-convex almost $CR$ structure on $T_{-1}M$, that we shall call the \emph{standard pseudo-convex structure}.
 The associated contact metric structure, also named \emph{standard contact metric structure}, is determined according to:
  \begin{equation*}
 \begin{aligned}
 &\varphi (X^O)=X^T, \quad \varphi (X^T)=-X^O, \quad  \varphi (\xi)=0,\\
 &g_{\eta}(\tilde{X}, \tilde{Y})=\frac{1}{4}\tilde{G}(\tilde{X}, \tilde{Y}), \quad g_{\eta}(\tilde{X}, \xi)=0,\quad g_{\eta}(\xi,\xi)=1,
 \end{aligned}
 \end{equation*}
 where  $X\in \mathfrak{X}(M)$ and $\tilde{X}, \tilde{Y}$ are any smooth sections of $H(T_{-1}M)$.

 \section{Contact metric $(k,\mu)$ structures on tangent hyperquadric bundles}
 
 In this section we prove our main results.
 
 \begin{theorem}\label{1}
  Let $(M,g)$ be a Lorentzian manifold.  Then $T_{-1}M$ is a locally symmetric pseudo-Hermitian manifold if and only if $(M,g)$ has constant sectional curvature.
 \end{theorem}
 
 \begin{proof}
 Suppose first that $(M,g)$ has constant sectional curvature. Let $t=(p,u)$ any point on $T_{-1}M$. We have that the linear mapping
 $$
 L:X\in T_pM\mapsto -X-2g_p(u,X)u\in T_pM,
 $$
 is an orthogonal transformation that preserves the Riemannian curvature tensor. Thus, there exists an isometry
 $$
 f:U\rightarrow U,
 $$
 where $U$ is an open neighborhood of $p$, such that $\di_p f=L$ (cf. \cite[Chapter 8]{ONeill}). Since $f$ is an isometry, we see that the induced diffeomorphism 
 $$
 F=\di f:TU\rightarrow TU
 $$
 satisfies:
 \begin{equation}
 \label{F-lifts}
  (\di F)_s(X^H_s)=(\di f_x(X))^H_{F(s)}, \quad (\di F)_s(X^V_s)=(\di f_x(X))^V_{F(s)},
 \end{equation}
 for every $X\in \mathfrak{X}(M)$ and $s=(x,v)\in T_{-1}M\cap TU$, hence $F$ is a local isometry with respect to the Sasaki metric $\tilde{G}$ on $TM$, preserving the almost complex structure $\tilde{J}$.  It follows that $F$ restricts to a local $CR$ diffeomorphism of $T_{-1}M$. Moreover,  \eqref{F-lifts} and
 \eqref{Reeb} imply that
 $(\di F)_s (\xi_s)=\xi_{F(s)}$,
yielding that $F$ is also a local isometry with respect to the Webster metric. Moreover, being $\di_p f=L$, we have:
\begin{equation*}
 (\di F)_t|_{H_t(T_{-1}M)}=-Id,
\end{equation*}
and thus $F$ is a local $CR$ symmetry at $t$.

Viceversa, if  $T_{-1}M$ is a locally symmetric pseudo-Hermitian manifold,  then in particular $(H(T_{-1}M),J)$ is a $CR$ structure and hence, by \cite[Theorem~1] {Perrone}, $(M,g)$ has constant sectional curvature.
 \end{proof}

 Now we determine the Boeckx invariant of $T_{-1}M$, where $M$ is a Lorentzian space form.
 
 \begin{theorem}
  Let $(M^{n+1},g)$ be a Lorentzian manifold with constant sectional curvature $c$. 
  Then, $T_{-1}M$ endowed with the standard contact metric structure is Sasakian if and only if $c=-1$.
  If $c\neq -1$ then $T_{-1}M$ is a non-Sasakian contact metric $(k,\mu)$-space, whose Boeckx invariant is:
  \begin{equation*}
   I=\frac{c-1}{|c+1|}.
  \end{equation*}

\end{theorem}

\begin{proof} Theorem~\ref{1} ensures that the standard contact metric structure $(\varphi,\xi,\eta,g_{\eta})$ of $T_{-1}M$ is a
contact metric $(k,\mu)$ structure (eventally a Sasakian one). 
In the following we compute the spectrum of  the symmetric operator $h$. Let $t=(p,u)\in T_{-1}M$ and $X\in T_pM$ such that $g_p(X,u)=0$. Then:
 \begin{equation}\label{h}
  \begin{aligned}
   2h(X^T)&=[\xi,\varphi X^T]-\varphi[\xi, X^T]\\
          &=-[\xi, X^O]-\tilde{J}[\xi, X^V+G(X^V,\N)\N],
  \end{aligned}
 \end{equation}
 where we are denoting again with $X$ any extension of the vector $X$. 
 Let $(x^i)$ be a local coordinate system on $M$ and $(q^i,v^i)$ the corresponding local coordinate system on $TM$. Since locally 
 $$
 \xi=-2 v^i (\frac{\partial}{\partial x^i})^H,
 $$
then using equation \eqref{brackets}, by a standard
computation (cf. also \cite{DileoLotta} in the Riemannian case),  we obtain:
 \begin{equation}\label{[xi,X]}
  [\xi, X^V]_t=2(X^H_t-(\nabla_uX)^V_t), \quad [\xi, X^H]_t=-2((\nabla_uX)^H_t-cX^V_t),
 \end{equation}
and hence \eqref{h} becomes:
 \begin{equation*}
  \begin{aligned}
   2h(X^T_t)&=2((\nabla_uX)_t^H-cX^V_t)-2(X^V_t+(\nabla_uX)_t^H)-\xi_t(G(X^V,\N))\tilde{J}\N_t\\
            &=-2(c+1)X^V_t-\frac{1}{4}\xi_t(G(X^V,\N))\xi_t.
  \end{aligned}
 \end{equation*}
It follows that $ \xi_t(G(X^V,\N))=0$, thus
\begin{equation}\label{eigenvalues}
\begin{aligned}
 &h(X^V_t)=-(c+1)X^V_t,\\
 &h(X^H_t)=h(-\varphi X^V_t)=\varphi h X^V_t=(c+1)X^H_t,
 \end{aligned}
\end{equation}
and the spectrum of the operator $h$ is $\{0, c+1, -(c+1)\}$. It follows that $T_{-1}M$ is Sasakian iff $c=-1$.

\medskip

Suppose $c\neq -1$. 
Let $t=(p,u)$ and $X\in T_pM$ such that $g_p(X,u)=0$. Then using the definition  \eqref{Pang inv def} of the Pang invariant and equation \eqref{[xi,X]},
we get, being $X^O$ a global section of $\mathcal{D}(c+1)$:
\begin{equation}\label{3}
\begin{aligned}
 \varPi_{\mathcal{D}(c+1)}(X^O_t,X^O_t) &= 2\di\eta([\xi,X^O]_t,X^O_t) \\
                                        &= 2g_{\eta}([\xi, X^H+\tilde{G}(X^V,\mathcal{N})\zeta]_t, X^T_t)\\
                                        &= 2g_{\eta}([\xi, X^H]_t+\xi (\tilde{G}(X^V,\mathcal{N}))\zeta_t,  X^T_t)\\
                                        &= 2g_{\eta}([\xi, X^H]_t, X^T_t)\\
                                        &= 4c g_{\eta}(X^V_t, X^T_t)\\
                                        &= 4cg_{\eta}(X^T_t, X^T_t).
 \end{aligned}
\end{equation}

\medskip

In particular if $c> -1$,
by equation \eqref{Pang inv 1}, we obtain
\begin{equation}\label{4}
 \varPi_{\mathcal{D}(c+1)}(X^O,X^O) = (2c+4-\mu)g_{\eta}(X^O,X^O).
\end{equation}
Thus comparing \eqref{3} and \eqref{4} we have that $\mu=4-2c$ and hence
$$
I_M=\frac{c-1}{c+1}.
$$
Finally suppose $c<-1$. Then by \eqref{Pang inv 2}:
\begin{equation}\label{5}
 \varPi_{\mathcal{D}(c+1)}(X^O,X^O) = (2c+4-\mu)g_{\eta}(X^O,X^O).
\end{equation}
Comparing equations \eqref{5} and \eqref{3} we obtain that $\mu=4-2c$ and 
$$
I_M=-\frac{c-1}{c+1}.
$$

\end{proof}

 \begin{corollary}
  Every non Sasakian contact metric $(k,\mu)$ space with Boeckx invariant $I\leqslant -1$ is locally equivalent, up to a $\mathcal{D}$-homothetic deformation, to the tangent hyperquadric bundle $T_{-1}M$ of a Lorentzian manifold $M$ with constant sectional curvature $c\le 0$, $c\not =-1$, endowed with its standard contact metrics structure.
 \end{corollary}

\vspace{3 mm}
Authors' addresses:

\medskip
\noindent 
Eugenia Loiudice

\vspace{1 mm}
\noindent Dipartimento di Matematica, Universit\`a di Bari Aldo Moro, 

\noindent Via Orabona 4, 70125 Bari, Italy

\noindent \emph{e-mail}: eugenia.loiudice@uniba.it

\vspace{3 mm}
\noindent 
Antonio Lotta

\vspace{1 mm}
\noindent Dipartimento di Matematica, Universit\`a di Bari Aldo Moro, 

\noindent Via Orabona 4, 70125 Bari, Italy

\noindent \emph{e-mail}: antonio.lotta@uniba.it


\begin{thebibliography}{}



\bibitem{Blair2010Book}
D.~E. Blair:
\newblock {{Riemannian geometry of contact and symplectic manifolds}},
  volume 203 of {\em {Progress in Mathematics}}.
\newblock Birkh\"auser, Boston Inc., Boston, MA, second edition, 2010.

\bibitem{Blair K P}
D.~E. Blair, T.~Koufogiorgos, and B.~J.~Papantoniou:
\newblock {\em Contact metric manifolds satisfying a nullity condition},
\newblock {Israel J. Math.} {\bf 91} (1995), 189–214.


\bibitem{Boeckx locally phi-symmetric}
E.~Boeckx:
\newblock {\em A class of locally $\varphi$-symmetric contact metric spaces},
\newblock {Arch. Math. (Basel)} {\bf 72} (1999), 466--472.


\bibitem{Boeckx}
E.~Boeckx:
\newblock {\em A full classification of contact metric $(k,\mu)$-spaces},
\newblock {Illinois J. Math.} {\bf 44} (2000), 212--219.



\bibitem{Boeckx-Cho}
E.~Boeckx and J. T.~Cho: \newblock{\em Pseudo-Hermitian symmetries}, \newblock{Israel J. Math.} {\bf 166} (2008), 125--145. 

\bibitem{Mino}
B.~Cappelletti Montano:
\newblock{\em The foliated structure of contact metric $(k,\mu)$-spaces}, \newblock{Illinois J. Math.} {\bf 53} (2009), no. 4, 1157--1172.

\bibitem{MinoDiTerlizzi}
B.~Cappelletti Montano and L.~Di Terlizzi: \newblock{ \em Contact metric $(k,\mu)$-spaces as bi-Legendrian manifolds}, \newblock{Bull. Aust. Math. Soc.}
{\bf 77} (2008), no. 3, 373--386.

\bibitem{DileoLotta}
G.~Dileo and A.~Lotta:
\newblock {\em A classification of spherical symmetric $CR$ manifolds.},
\newblock {Bull. Aust. Math. Soc.} {\bf 80} (2009), 251--274.

\bibitem{Perrone Dragomir}
S.~Dragomir and D.~Perrone:
\newblock {\em On the geometry of tangent hyperquadric bundles: $CR$ and pseudoharmonic vector fields},
\newblock {Ann. Global Anal. Geom.} {\bf 30} (2006), 211--238.



\bibitem{LibroDragomirTomassini}
S.~Dragomir and G.~Tomassini:
\newblock {Differential geometry and analysis on $CR$ manifolds},
{\em {Progress in Mathematics}}.
\newblock Birkh\"auser, Boston Inc., Boston, 2006.



\bibitem{Kaup Zaitsev}
 W.~Kaup and D.~Zaitsev:
\newblock {\em On symmetric Cauchy-Riemann manifolds},
\newblock {Adv. Math.} {\bf 149} (2000), 145--181.

\bibitem{ONeill}
B.~O'Neill: Semi-Riemannian geometry. With applications to relativity. Pure and Applied Mathematics, 103. Academic Press, Inc., New York, 1983.

\bibitem{Pang}
M.~Y. Pang:
\newblock {\em The structure of Legendre foliations},
\newblock {Trans. Amer. Math. Soc.} {\bf 320}  (1990), 417--455.


\bibitem{Perrone}
D.~Perrone:
\newblock {\em On the standard nondegenerate almost $CR$ structure of tangent hyperquadric bundles},
\newblock { Geom. Dedicata} {\bf }  (2016), doi:10.1007/s10711-016-0167-z.

\end{thebibliography}
\end{document}